\pdfoutput=1

\documentclass[final]{siamltex}

\usepackage{amsmath,amssymb,amsfonts,cite,color}

\usepackage{mathtools,enumerate}

\newcommand{\eps}{\varepsilon}
\newcommand{\qed}{\endproof}

\newtheorem{Theorem}{theorem}
\newtheorem{remark}{remark}

\numberwithin{equation}{section}

\begin{document}

\title{\bf Eventual self-similarity of solutions for the diffusion
  equation with nonlinear absorption and a point source}

\author{Peter V. Gordon\thanks{Department of Mathematics, University
    of Akron, Akron, Ohio 44325, USA ({\tt pgordon@uakron.edu)}} \and
  Cyrill B. Muratov\thanks{Department of Mathematical Sciences, New
    Jersey Institute of Technology, University Heights, Newark, NJ
    07102, USA ({\tt muratov@njit.edu})}}

\maketitle

\begin{abstract}
  This paper is concerned with the transient dynamics described by the
  solutions of the reaction-diffusion equations in which the reaction
  term consists of a combination of a superlinear power-law absorption
  and a time-independent point source. In one space dimension,
  solutions of these problems with zero initial data are known to
  approach the stationary solution in an asymptotically self-similar
  manner. Here we show that this conclusion remains true in two space
  dimensions, while in three and higher dimensions the same conclusion
  holds true for all powers of the nonlinearity not exceeding the
  Serrin critical exponent. The analysis requires dealing with
  solutions that contain a persistent singularity and involves a
  variational proof of existence of ultra-singular solutions, a
  special class of self-similar solutions in the considered problem.
\end{abstract}

\begin{keywords} 
  Self-similarity, diffusion-absorption, source-sink models, morphogen
  gradients.
\end{keywords}

\begin{AMS}
  35C06, 35K61, 35B40, 35Q92
\end{AMS}

\pagestyle{myheadings} \thispagestyle{plain} \markboth{P. V. GORDON
  AND C. B. MURATOV}{EVENTUAL SELF-SIMILARITY FOR
  DIFFUSION-ABSORPTION}

\section{Introduction}\label{sec:intro}

Many problems of mathematical biology can be modeled by
reaction-diffusion equations containing strongly localized source
terms. For example, in embryonic development locally produced
molecules called morphogens spread through the developing tissue,
producing graded concentration profiles that provide positional
information guiding morphogenesis \cite{ms02}.  The simplest mechanism
of morphogen gradient formation involves diffusion and degradation of
morphogen molecules throughout the tissue
\cite{rmss:dc06,llnw09,othmer09,wkg09,sb12}.  Typically, the
degradation is a nonlinear process due to the presence of regulatory
feedbacks \cite{ab06}. In one of the generic regulatory mechanisms the
degradation rate is an increasing function of the morphogen
concentration, which in the simplest case can be modeled by a power
law \cite{eldar03}.  After a suitable non-dimensionalization, this
mechanism gives rise to the following canonical reaction-diffusion
model \cite{eldar03,kpbkbjg07,twh07,ybnrpssb09,othmer09,llnw09,wkg09}:
\begin{eqnarray}\label{eq:u}
  u_t=\Delta u-u^p+\alpha \delta(x).
\end{eqnarray}
Here, $u=u(x,t)\in \mathbb{R}^+$ and represents the morphogen
concentration, $x\in\mathbb{R}^d$ is the spatial coordinate, $t\in
\mathbb{R}^+$ is time, $\delta(x)$ is the spatial Dirac
delta-function, and $p>1$ and $\alpha>0$ are the order of the
degradation reaction and the source strength, respectively. Note that
although our interest in this type of equations arose in the context
of morphogen gradients
\cite{gsbms:pnas11,mgs:pre11,gm:nhm12,gms:jcp13}, this equation is
quite general and may also arise, for example, in the studies of
intracellular calcium dynamics \cite{smith01,saftenku12}, the dynamics
of photo-generated charge carriers in semiconductors
\cite{smith,ko:book}, or the kinetics of a non-conserved order
parameter in Ginzburg-Landau models of phase transitions in the
presence of heterogeneities \cite{hohenberg77,levanyuk79}.

In this paper we are interested in the transient dynamics described by
\eqref{eq:u} with zero initial data, $u(x,0)=0$ (for the precise
definition of the solution, see section \ref{sec:main}).  Recently, we
studied this problem in one space dimension, $d=1$
\cite{gm:nhm12,mgs:pre11}. In this case we proved that the solution of
this problem is monotone increasing in $t$ for any $x$ and converges
to the unique stationary solution.  Moreover, we showed that the ratio
of the solution of the initial value problem to the stationary
solution, written as a function of time $t$ and the parabolic
similarity variable $x/\sqrt{t}$, approaches some limit profile
independent of $\alpha$ when $t\to\infty$.  This has to do with the
existence of an {\em ultra-singular solution} for problem
\eqref{eq:u}, namely a function $U: \mathbb{R}^d\backslash\{0\}\times
\mathbb{R}^+ \to \mathbb{R}^+$ verifying
\begin{eqnarray}\label{eq:using}
  \left\{
\begin{array}{ll}
  U_t=\Delta U-U^p & (x,t)\in\mathbb{R}^d\backslash\{0\}\times
  \mathbb{R}^+, \\ 
  \displaystyle\lim_{t\to 0} U(x,t)=0 & x
  \in\mathbb{R}^d\backslash\{0\}, \\ 
  \displaystyle\lim_{|x|\to 0} |x|^{\frac{2}{p-1}}U(x,t)>0  \qquad &
  t\in\mathbb{R}^+. 
\end{array}\right.
\end{eqnarray}
Furthermore, this function is self-similar in the sense that
$|x|^{\frac{2}{p-1}}U(x,t)$ can be written as a function of the
similarity variable only. In addition, this function is, equivalently,
the limit of the solutions of \eqref{eq:u} with zero initial data as
$\alpha\to \infty$.

Our choice of terminology is motivated by another class of
self-similar solutions, called {\em very singular solutions}, which
arise in a closely related problem that is given by \eqref{eq:u} with
$\alpha = 0$ and the initial data in the form of a constant multiple
of a delta-function \cite{galaktionov85,brezis86,escobedo87}. These
solutions satisfy \eqref{eq:using}, except for the last condition, and
are instead smooth for all $t > 0$, blow up faster than the heat
kernel at the origin as $t \to 0$ and go to zero uniformly as
$t\to\infty$. Moreover, these solutions are, in a suitable sense, the
long time attractors of the solutions of the above problem
\cite{kamin85}.  In contrast, the ultra-singular solutions constructed
by us contain a persistent singularity at the origin for all $t > 0$,
as can be seen from the third condition in \eqref{eq:using}, and
approach from below the unique stationary solution as
$t\to\infty$. Let us also note that existence of very singular
solutions depends delicately on the power $p$ of the nonlinearity and
the space dimension $d$. In fact, non-trivial solutions of the problem
considered in \cite{brezis86} exist if and only if $p < p_c$, where
$p_c := (d + 2)/d$ is the Fujita critical exponent \cite{brezis83}. In
particular, in one space dimension very singular solutions exist if
and only if $p < 3$. At the same time, as we showed in
\cite{gm:nhm12}, ultra-singular solutions exist for all $p > 1$ in one
space dimension.

Coming back to our problem, it is natural to expect that the solutions
of \eqref{eq:u} with zero initial data would exhibit the same type of
behavior in dimensions $d>1$, namely, that they approach the unique
stationary solution of \eqref{eq:u} from below as $t \to \infty$, and
do so in an asymptotically self-similar fashion. However, in going to
arbitrary dimensions one encounters the difficulty that for $d \geq 2$
non-negative stationary solutions that satisfy \eqref{eq:u}
classically away from the origin exist if and only if $p < p^*$
\cite{brezis87,brezis80,veron81,benilan04} (for a recent overview of
results on nonlinear elliptic equations involving measures, see
\cite{ponce12}), where
\begin{eqnarray}\label{eq:p*}
  p^* := \left\{
\begin{array}{ll}
  \infty & d=2, \\
  \frac{d}{d-2} & d\ge 3,
\end{array}
\right.
\end{eqnarray}
is often referred to as the Serrin critical exponent.  Therefore, one
can only hope to extend our one-dimensional results to the case $d> 1$
when $p<p^*$. In fact, the latter condition is also necessary for the
existence of distributional solutions of \eqref{eq:u} with a
persistent singularity at the origin \cite{baras84}.

In this paper, we study equation \eqref{eq:u} with zero initial
condition, $d\ge 2$ and $1<p<p^*$.  Under these assumptions we prove
that the one-dimensional picture described above extends to higher
dimensions.  The main difference with the one-dimensional case is that
all solutions of \eqref{eq:u} with zero initial data are unbounded
near the origin for all times $t>0$. Note that singular solutions of
\eqref{eq:u} (with $\alpha \geq 0$) have been considered in the
literature in a variety of contexts
\cite{brezis83,brezis86,baras84,oswald88,kamin85}. In particular,
Baras and Pierre developed a general existence theory for solutions of
\eqref{eq:u} in which the delta-function in the right-hand side is
replaced by a general bounded Radon measure on $\mathbb R^d \times
\mathbb R^+$ \cite{baras84}. They provided a necessary and sufficient
condition for existence in terms of parabolic capacity.

Although the framework of \cite{baras84} can be straightforwardly
extended to the problem under consideration, in view of the special
form of the source term in the right-hand side of \eqref{eq:u} we
chose to give a self-contained proof of existence, uniqueness,
regularity and some qualitative properties of solutions of
\eqref{eq:u} with zero initial data. These results are presented in
Theorem \ref{t:eu}.  We then proceed to establish existence of a
self-similar solution that satisfies \eqref{eq:using}.  The proof is
an adaptation of our variational proof in \cite{gm:nhm12} for the
one-dimensional case (for a related variational proof in a different
context, see also \cite{escobedo87}). The result is contained in
Theorem \ref{t:exun}.  Finally, our main result is on the asymptotic
self-similar behavior of solutions of \eqref{eq:u} with zero initial
data given in Theorem \ref{t:2}.  There we prove that the long-time
limit of solutions can be characterized by the self-similar solution
constructed in Theorem \ref{t:exun}.

From the point of view of the applications, our results apply in two
space dimensions irrespectively of the power $p$ of the
nonlinearity. In particular, they indicate that the approach of the
morphogen concentration to the steady state in a developing
epithelium, which to a first approximation is a two-dimensional layer
of cells, is asymptotically self-similar and, therefore, exhibits
robustness in the case of a point source in the plane. Such robustness
was previously demonstrated for these types of problems in one space
dimension \cite{eldar03,mgs:pre11,gsbms:pnas11}. Our results also
apply in the case of three space dimensions and a physically important
case of a second-order degradation reaction, $p = 2$. At the same
time, our results break down in the case $d = 3$ and $p = 3$,
corresponding to, e.g., the Ginzburg-Landau equation for the
non-conserved order parameter in the presence of a localized
heterogeneity \cite{levanyuk79}.

Our paper is organized as follows. In Sec. \ref{sec:main}, we
introduce the notation used throughout the paper, present a few
auxiliary facts and state the main results. Sections \ref{s:3},
\ref{s:4} and \ref{s:5} are then devoted to the proofs of each the
three theorems in Sec. \ref{sec:main}, respectively.

\section{Preliminaries and the main results}
\label{sec:main}
 
In this section, we introduce the notations, collect a number of known
results that will be useful throughout the paper and state our main
theorems.

We begin with a discussion of the stationary solutions for
\eqref{eq:u}. By a stationary solution for a given $\alpha > 0$, we
mean a non-negative function $v_\alpha \in L^p(\mathbb R^d)$ which
satisfies
\begin{eqnarray}\label{eq:v}
  -\Delta v_\alpha+v_\alpha^p=\alpha \delta(x) \qquad  \text{in} \quad
  \mathcal D'(\mathbb{R}^d). 
\end{eqnarray} 
Solutions of \eqref{eq:v} enjoy the following properties:

\begin{enumerate}[i)]
\item For each $\alpha>0$, $d \geq 2$ and $1<p<p^*$, problem
  \eqref{eq:v} has a unique positive solution that belongs to
  $C^{\infty}(\mathbb{R}^d\backslash\{0\})$, is radially-symmetric, and
  behaves as
\begin{align}
  \label{eq:va0}
  v_\alpha(x) \simeq \alpha \Phi(x), \qquad |x| \ll 1,
\end{align}
where
\begin{align}
  \label{eq:Gam}
  \Phi(x) :=
  \begin{dcases}
    {1 \over 2 \pi} \ln {1 \over |x|}, & d = 2, \\
    {1 \over (d - 2) |\mathbb S^{d-1}|} \cdot {1 \over |x|^{d-2}}, & d
    \geq 3,
  \end{dcases}
\end{align}
is the fundamental solution of the Laplace's equation.

\item $v_{\alpha}(x)$ is an increasing function of $\alpha$ for each
  $x \not= 0$ fixed and approaches from below the function
  $v_{\infty}(x)$ as $\alpha\to\infty$, where
\begin{eqnarray}\label{eq:vinf}
  v_{\infty}(x) :=\frac{c(p,d)}{|x|^{\frac{2}{p-1}}}, \quad
  c(p,d)=\left[\frac{2}{p-1}\left(\frac{2p}{p-1} -
      d \right)\right]^{\frac{1}{p-1}}.   
\end{eqnarray}
Furthermore, $v_\infty(x)$ is the only classical solution of
\eqref{eq:va0} on $\mathbb R^d \backslash \{0\}$ which grows faster
than $\Phi(x)$ as $|x| \to 0$.
\end{enumerate}

\medskip

\noindent For these and other results related to the solutions of
\eqref{eq:va0}, we refer the reader to
\cite{brezis87,brezis80,veron81,benilan04,ponce12} and further
references therein.

We next turn to assigning the meaning to the solutions of \eqref{eq:u}
with zero initial data. As these solutions are expected to exhibit a
singularity of the type $\alpha \Phi(x)$ near the origin, they should
be understood in an appropriate distributional sense. To illustrate
this point, let us first consider the {\em linearized} version of
\eqref{eq:u}. Extending the solution by zero for $t \leq 0$ and
setting $\alpha = 1$ for simplicity, we are lead to the following
equation:
\begin{align}
  \label{eq:I}
  I_t = \Delta I + \delta(x) \theta(t) \qquad \text{in} \quad \mathcal
  D'(\mathbb R^{d+1}),
\end{align}
where $\theta(t)$ is the Heaviside step function. The solution of this
equation reads
\begin{align}
  \label{eq:Isol}
  I(x, t) = \int_0^t {\theta(t) \over [4 \pi (t - s)]^{d/2}} e^{-
    {|x|^2 \over 4 (t - s)}} ds = \frac{1}{4 \pi^{d/2}} {\theta(t)
    \over |x|^{d - 2}} \Gamma \left( \frac{d}{2} - 1, {|x|^2 \over 4
      t} \right),
\end{align}
where $\Gamma(a, x)$ is the incomplete Gamma-function
\cite{abramowitz}. Note that $I(x, t)$ is a monotonically increasing
function of $t$ that approaches $\Phi(x)$ from below for each $x \not
= 0$ as $t \to \infty$ when $d \geq 3$, while it blows up
logarithmically for all $x \not = 0$ as $t \to \infty$ in the case $d
= 2$. Also note that for $t > 0$ we have the following asymptotic
behavior of $I(x, t)$:
\begin{align}
  \label{eq:Iasymp}
  I(x, t) \simeq
  \begin{dcases}
    t^{1-\frac{d}{2}} \Phi(x / \sqrt{t}), & |x| \ll \sqrt{t}, \\
    2^{2-d} \pi^{-d/2} t^{2-\frac{d}{2}} e^{-\frac{|x|^2}{4 t}}
    |x|^{-2}, & |x| \gg \sqrt{t}.
  \end{dcases}
\end{align}
In particular, a straightforward calculation shows that $I(\cdot, t)
\in L^p(\mathbb R^d)$ for every $t > 0$, provided that $1 \leq p <
p^*$. This dictates that for those values of $p$ the nonlinear term in
\eqref{eq:u} is expected to be ``dominated'' by the delta-function
near the origin. The latter observation is key to the well-posedness
of the initial value problem for \eqref{eq:u}.

We next introduce the definition of solutions of \eqref{eq:u} with
zero initial data as distributions for which the nonlinearity in
\eqref{eq:u} also makes sense.

\begin{definition}
  \label{d:solu}
  We call $u(x, t)$ a solution of \eqref{eq:u} with zero initial data,
  if for any $T > 0$ the map $t \mapsto u(\cdot, t)$ belongs to $C([0,
  T]; L^p(\mathbb R^d))$, $u(\cdot, 0) = 0$ and
  \begin{align}
    \label{eq:ud}
    \int_0^T \int_{\mathbb R^d} u \left( \varphi_t + \Delta \varphi -
      |u|^{p-1} \varphi \right) dx \, dt + \alpha \int_0^T \varphi(0,
    t) dt = 0 \qquad \forall \varphi \in C^\infty_c(\mathbb R^d \times
    (0, T)).
  \end{align}
\end{definition}

Our first result concerns existence, uniqueness, regularity and
qualitative properties of positive solutions from Definition
\ref{d:solu} (see also the general framework in \cite{baras84}).

\begin{Theorem}\label{t:eu}
  For each $\alpha>0$, $d \geq 2$ and $1<p<p^*$, there is a unique
  positive solution $u(x, t)$ of \eqref{eq:u} with zero initial data
  in the sense of Definition \ref{d:solu}. Moreover, it is
  radially-symmetric, solves
  \begin{align}
    \label{eq:u0}
    u_t = \Delta u - u^p
  \end{align}
  classically for all $(x,t) \in \mathbb{R}^d\backslash\{0\}\times
  \mathbb R^+$ and obeys $u(x, t) \simeq \alpha \Phi(x)$ for all $t >
  0$ and $|x| \ll 1$.  In addition, for each $x \not = 0$ the map $t
  \mapsto u(x,t)$ is non-decreasing, and $u(x, t) \to v_{\alpha}(x)$
  from below as $t\to \infty$.
\end{Theorem}

Our next result concerns self-similar solutions of \eqref{eq:using}
that are constructed via the similarity ansatz
\begin{eqnarray}\label{eq:simanz}
  U(x,t) = v_{\infty}(x) \phi(\zeta), \quad \zeta=\ln\left(
    \frac{|x|}{\sqrt{t}}\right).
\end{eqnarray}
Here $0 \leq \phi(\zeta) \leq 1$ is some unknown function, which will
be referred to as the {\em self-similar profile}.  Substituting the
similarity ansatz from \eqref{eq:simanz} into \eqref{eq:using}, after
some algebra we obtain the following equation for the self-similar
profile $\phi$:
\begin{align}\label{eq:phiz}
  \phi^{\prime\prime}+\left(\frac{e^{2\zeta}}{2}-\frac{p+3}{p-1}+d-1\right)
  \phi^{\prime}+
  \frac{2}{p-1}\left(\frac{p+1}{p-1}-d+1 \right)(\phi-\phi^p)=0, \quad
  \zeta \in \mathbb R,
\end{align}
supplemented with the limit behavior
\begin{align}
  \label{eq:phibcz1}
  \lim_{\zeta \to -\infty} \phi(\zeta) = 1, \qquad \lim_{\zeta \to
    -\infty} {d \phi(\zeta) \over d \zeta} = 0,  \\ 
  \lim_{\zeta \to +\infty} \phi(\zeta) = 0, \qquad
  \lim_{\zeta \to +\infty} 
  {d \phi(\zeta) \over d \zeta} = 0.   \label{eq:phibcz2}
\end{align}

Solutions of \eqref{eq:phiz} satisfying \eqref{eq:phibcz1} and
\eqref{eq:phibcz2} will be sought as weak solutions belonging, after
subtracting a function $\eta \in C^\infty(\mathbb R)$ that obeys
\begin{align}
  \label{eq:eta}
  \eta'(\zeta) \leq 0, \quad \zeta \in \mathbb R, \qquad \eta(\zeta) =
  1, \quad \zeta \in (-\infty, 0], \qquad \eta(\zeta) = 0, \quad \zeta
  \in [1, +\infty),
\end{align}
 to the
weighted Sobolev space $H^1(\mathbb R, d \mu)$, which is defined as
the completion of the family of smooth functions with compact support
with respect to the Sobolev norm 
\begin{eqnarray}
  \label{eq:H1}
  ||w||_{H^1(\mathbb R, d \mu)}^2  :=  ||w_\zeta||_{L^2(\mathbb R, d
    \mu)}^2 +  ||w||_{L^2(\mathbb R, d \mu)}^2,
\end{eqnarray}
where $||w||_{L^2(\mathbb R, d \mu)}^2 := \int_\mathbb{R} w^2(\zeta) d
\mu(\zeta)$, and the measure $\mu$ is defined as
\begin{eqnarray}
  \label{eq:muu}
  d \mu(\zeta) := \rho(\zeta) d \zeta, \qquad \rho(\zeta) :=  
  \exp\left\{\frac{e^{2\zeta}}{4}-\left(\frac{p+3}{p-1}-d+1\right) \zeta 
  \right\}.
\end{eqnarray}
This setting allows to view \eqref{eq:phiz} as the Euler-Lagrange
equation of a certain energy functional and prove existence of
solutions via the direct method of calculus of variations (compare
also with \cite{escobedo87}). We have the following result.

\begin{Theorem}\label{t:exun}
  For each $d \geq 2$ and $1 < p < p^*$, there exists a unique weak
  solution $\phi \in H^1_{loc}(\mathbb R)$ of \eqref{eq:phiz} such
  that $\phi - \eta \in H^1(\mathbb R, d \mu)$ and $0 \leq \phi \leq
  1$. Furthermore, $\phi \in C^\infty(\mathbb R)$, satisfies
  \eqref{eq:phiz} classically, and $0 < \phi < 1$. In addition, $\phi$
  is strictly decreasing, satisfies \eqref{eq:phibcz1} and
  \eqref{eq:phibcz2}, and has the following asymptotic behavior:
\begin{eqnarray}\label{eq:asympt}
  && \phi(\zeta)\sim
  \exp\left\{ -\frac{e^{2\zeta}}{4}+\left(\frac{5-p}{p-1}-d+1
    \right)\zeta \right\},  \qquad \zeta\to+\infty . 
\end{eqnarray}
\end{Theorem}


Finally, the main result of this paper is the following.

\begin{Theorem}\label{t:2}
  For $\alpha>0$, $d \geq 2$ and $1 < p < p^*$, let $u(x, t)$ be the
  solution of \eqref{eq:u} with zero initial data, and define
  \begin{eqnarray}\label{eq:tm00}
    F(\zeta,t):=\frac{u(x,t)}{v_{\alpha}(x)}, \qquad
    \zeta=\ln\left(\frac{|x|}{\sqrt{t}}\right).  
  \end{eqnarray}
  Then
  \begin{eqnarray}\label{eq:tm01}
    F(\zeta,t)\to \phi(\zeta) \quad \mbox{as} \quad t\to \infty, 
  \end{eqnarray}
  where $\phi(\zeta)$ is as in Theorem \ref{t:exun}.
\end{Theorem}

\section{Proof of Theorem \ref{t:eu}}
\label{s:3}

In this section, we prove existence of distributional solutions of
\eqref{eq:u} with zero initial data, as well as their uniqueness,
regularity, asymptotic behavior at the origin and monotonic approach
from below to the solution of \eqref{eq:v}. As we already noted in the
introduction, existence of these singular solutions can be treated
within the general framework developed in \cite{baras84}. For the sake
of completeness, we give a self-contained proof that uses the special
form of the measure appearing in \eqref{eq:u} and follows the ideas
used in the analysis of the elliptic case (for an overview, see, e.g.,
\cite{ponce12}).

To prove existence of solutions of \eqref{eq:u} with zero initial data
in the sense of Definition \ref{d:solu}, we mollify the delta-function
and consider for each $n \in \mathbb N$ the solution $u = u_n(x, t)$
that vanishes at $t = 0$ of the equation
\begin{align}
  \label{eq:un}
  u_t = \Delta u - f_n(u) + \alpha g_n(x), \qquad (x, t) \in \mathbb
  R^d \times \mathbb R^+,
\end{align}
where $g_n(x) = n^d g(n x)$ for some non-negative, radially-symmetric
function $g \in C^\infty_c(\mathbb R^d)$ supported on a unit ball and
satisfying $\int_{\mathbb R^d} g(x) dx = 1$. Also, here $f_n \in
\text{Lip}(\mathbb R)$ is the extended and truncated nonlinearity,
namely,
\begin{align}
  \label{eq:fn}
  f_n(u) := 
  \begin{cases}
    0, & u < 0, \\
    u^p, & 0 \leq u \leq \bar u_n, \\
    \bar u_n^p, & u > \bar u_n ,
  \end{cases}
  \qquad \bar u_n := \left( \alpha \| g_n \|_{L^\infty(\mathbb R^d)}
  \right)^{1/p}.
\end{align}
Short-time existence of a unique solution $u_n$ of \eqref{eq:un} in
the class of bounded uniformly continuous functions then follows from
the classical theory (see, e.g., \cite[Chapter 15]{taylor3},
\cite{lunardi}). Moreover, since $u = 0$ and $u = \bar u_n$ are a sub-
and a super-solution for \eqref{eq:un}, the solution is, in fact,
global in time, and we have $f_n(u_n) = u_n^p$. By standard parabolic
theory, this solution is smooth for all $t \geq 0$. Also, by
construction the solution is radially-symmetric. In addition, by
comparison principle \cite{protter}, for all $t > 0$ the solution is
positive and monotonically increasing in $t$ for each $x \in \mathbb
R^d$. In particular, we have $u_n(x, t) \to v_{\alpha,n}(x) > 0$ from
below for all $x \not= 0$ as $t \to \infty$, where $v_{\alpha,n}$
solves
\begin{align}
  \label{eq:van}
  -\Delta v_{\alpha,n} + v_{\alpha,n}^p = \alpha g_n(x).
\end{align}
Furthermore, it is easy to see that $v_{\alpha,n} \leq
\overline{v}_{\alpha,n}$, where $\overline{v}_{\alpha,n}(x) := C
|x|^{-{2 \over p - 1}}$, for some $C > 0$ depending only on $d$, $p$,
$\alpha$ and $\|g\|_{L^\infty(\mathbb R^d)}$. Indeed, choosing $C$
sufficiently large, we have that the left-hand side of \eqref{eq:van}
is bounded from below by $\frac12 C^p |x|^{-{2 p \over p - 1}} \geq
\alpha \| g \|_{L^\infty(\mathbb R^d)}|x|^{-d} \geq \alpha g_n(x)$ for
all $|x| \leq {1 \over n}$, as long as $p < p^*$. At the same time,
trivially $\frac12 C^p |x|^{-{2 p \over p - 1}} \geq 0 = \alpha
g_n(x)$ for all $|x| > {1 \over n}$. Hence $\overline{v}_{\alpha,n}$
is a super-solution, and the conclusion follows by maximum principle
\cite{pucci}.

We now show that 
\begin{align}
  \label{eq:unb}
  u_n(x, t) \leq \alpha C_d I(x, 1 + 2 t) \qquad \forall (x, t) \in
  \mathbb R^d \times \mathbb R^+,
\end{align}
for some $C_d > 0$ depending only on the dimension and the choice of
$g$, where $I(x, t)$ is defined in \eqref{eq:Isol}. Indeed, by
comparison principle $u_n$ can be estimated from above by the solution
of the linearized equation:
\begin{align}
  \label{eq:unb2}
  u_n(x, t) \leq \alpha \int_{\mathbb R^d} I(x - y, t) g_n(y) dy
  \qquad \forall (x, t) \in \mathbb R^d \times \mathbb R^+.
\end{align}
Therefore, recalling the definitions of $g_n$ and $I$, we have
\begin{align}
  \label{eq:unb3}
  u_n(x, t) & \leq \alpha I(x, 1 + 2 t) \int_{\mathbb R^d} {I(x - y,
    t)
    \over I(x, 1 + 2 t)} g_n(y) dy \notag \\
  & \leq \alpha I(x, 1 + 2 t) \| g\|_{L^\infty(\mathbb R^d)}
  \int_{B_1(0)} {I(\tilde x - \tilde y, n^2 t) \over I(\tilde x, n^2
    (1 + 2 t))} d\tilde y,
\end{align}
where $\tilde x := n x$.  It is then not difficult to see from
\eqref{eq:Iasymp} that the integral in the right-hand side of
\eqref{eq:unb3} is bounded independently of $\tilde x$, $t$ and $n$,
which yields the claim.

The estimate in \eqref{eq:unb} and the fact that $p < p^*$ imply that
$u_n(\cdot, t)$ are uniformly bounded in $L^p(\mathbb R^d)$ for all $t
\in [0, T]$, with $T > 0$ arbitrary. Therefore, the right-hand side of
\eqref{eq:un} is uniformly bounded in $L^1(\mathbb R^d)$ for each $t
\in [0, T]$, and by parabolic theory \cite[Chapter 15]{taylor3} we
also have that $u_n(\cdot, t)$ is uniformly bounded in
$W^{1,q}(\mathbb R^d)$ for any $1 \leq q < d/(d - 1)$. Choosing $q$
sufficiently close to $d/(d-1)$, which is, again, possible for $p <
p^*$, and passing to a subsequence (not relabeled), by
Rellich-Kondrachov theorem we then conclude that $u_n(\cdot, t) \to
u(\cdot, t)$ strongly in $L^p_{loc}(\mathbb R^d)$ as $n \to \infty$
for every $t \in [0, T]$. In fact, the bound in \eqref{eq:unb} implies
that convergence is in $L^p(\mathbb R^d)$. In addition, by Sobolev
embedding and parabolic theory we have uniform H\"older continuity of
$u_n(\cdot, t)$ in $L^p(\mathbb R^d)$ \cite[Chapter 4]{lunardi} and,
hence, upon extraction of a subsequence we have $u_n \to u$ in $C([0,
T]; L^p(\mathbb R^d))$. Thus, we can pass to the limit in the
distributional formulation of \eqref{eq:un} to obtain that $u$
satisfies \eqref{eq:ud}. By construction, $u(\cdot, t)$ is radial. By
the same line of arguments, we also have $v_{\alpha,n} \to v_\alpha$
strongly in $L^p(\mathbb R^d)$ as $n \to \infty$.

Uniqueness of solutions of \eqref{eq:ud} follows from parabolic
theory, noting that the difference $w(x, t)$ of any two solutions of
\eqref{eq:ud} satisfies
\begin{align}
  \label{eq:ww}
  w_t = \Delta w + k(x, t) w \quad \text{in} \quad \mathcal D'(\mathbb
  R^d \times \mathbb R^+),
\end{align}
where $k(\cdot, t)$ is uniformly bounded in $L^{p'}(\mathbb R^d)$,
with $p' = p/(p - 1)$, for all $t \in [0, T]$, and is, hence,
zero. Note that uniqueness also implies that $u$ is the full limit of
$u_n$ as $n \to \infty$. Regularity outside the origin is also
standard in view of the estimate in \eqref{eq:unb}, which implies that
$u(\cdot, t)$ is uniformly bounded on $\mathbb R^d \backslash B_R(0)$
for any $R > 0$. Furthermore, by the bootstrap argument $u_n(\cdot,
t)$ converges to $u(\cdot, t)$ uniformly in $\mathbb R^d \backslash
B_R(0)$ as $n \to \infty$. In particular, the limit $u(x, t)$ of
$u_n(x, t)$ is non-decreasing in $t$ for all $x \not = 0$ and
approaches some stationary solution $v(x) \leq \limsup_{n \to \infty}
v_{\alpha,n}(x)$ as $t \to \infty$. At the same time, by elliptic
regularity we also have $v_{\alpha,n}(x) \to v_\alpha(x)$ for all $x
\in \mathbb R^d \backslash \{0\}$ as $n \to \infty$. It then follows
that $v(x) \leq v_\alpha(x)$.

Finally, the asymptotic behavior of $u(\cdot, t)$ near the origin
follows from \eqref{eq:Iasymp} and the fact that $u - \alpha I$ solves
forced heat equation with the forcing term of order $|x|^{-s}$ near
the origin for some $0 < s < d$ and is, therefore, less singular than
$\Phi$ there.  Combining this with monotonicity of $u(x, t)$ in $t$
and \eqref{eq:Iasymp}, we then conclude that $u(x, t) \simeq \alpha
\Phi(x)$ for $|x| \ll 1$ uniformly in $t$, which implies that $v =
v_\alpha$. \qed

\section{Proof of Theorem \ref{t:exun}}
\label{s:4}

Observe that \eqref{eq:phiz} is the Euler-Lagrange equation for the
energy functional
\begin{align}
  \label{eq:E}
  \mathcal E[\phi] := & \int_{\mathbb R} \Bigg\{ \frac12 \left( {d
      \phi \over d \zeta} \right)^2 + \left( {1 \over p - 1} - {d - 1
      \over p + 1} \right) \eta \notag \\
  & - \left( { 1 \over (p - 1)^2} - {d - 1 \over p^2 - 1} \right)
  \phi^2 \left( p + 1 - 2 \phi^{p-1} \right) \Bigg\} d\mu(\zeta),
\end{align}
which is well-defined for all $\phi - \eta \in H^1(\mathbb R, d
\mu)$. This energy can be written as
\begin{align}
  \label{eq:EE}
  \mathcal E[\phi] = \int_{\mathbb R} \Bigg\{ \frac12 \left( {d \phi
      \over d \zeta} \right)^2 + \left( {p + 1 \over p - 1} - K(d)
  \right) \left[ {\eta \over p + 1} - {\phi^2 \left( p + 1 - 2
        \phi^{p-1} \right) \over p^2 - 1} \right] \Bigg\} d\mu(\zeta),
\end{align}
where
\begin{eqnarray}
K(d)=\left\{
\begin{array}{ll}
  1 & d=2, \\
  \frac{p^*+1}{p^*-1} & d\ge 3.
\end{array}
\right.
\end{eqnarray}
Clearly, the term multiplying the square bracket in \eqref{eq:EE} is
positive for all $p < p^*$. This means that we can apply the proof in
\cite[Theorem 2.1]{gm:nhm12} verbatim to prove existence and qualitative properties
of solutions of \eqref{eq:phiz} in the considered class of
functions. In particular, the decay estimate follows as in
\cite[Remark 1]{gm:nhm12}. \qed

\section{Proof of Theorem \ref{t:2}}
\label{s:5}

We start with the following simple observation.  Letting $R>0$,
consider the problem
\begin{eqnarray}\label{eq:w}
\left\{
\begin{array}{ll}
  w_t=w_{rr} + {d - 1 \over r} w_r -w^p &
  (r,t)\in (R, \infty) \times (0,\infty), \\ 
  w(R,t)=u(R, t)& t\in (0,\infty),\\
  w(r,0)=0 & r \in(R, \infty),
\end{array}\right.
\end{eqnarray}
where, with some abuse of notation, for any $r > 0$ we take $u(r, t)$
to be the value of the solution $u(x,t)$ of \eqref{eq:u} with $|x| =
r$.  By regularity of $u(R, t)$ (see Theorem \ref{t:eu}) and standard
parabolic theory, problem \eqref{eq:w} has a unique classical
solution. Therefore, by uniqueness we have $w(r,t)=u(r,t)$ for $r\ge
R$ and $t\ge 0$.  The advantage of considering problem \eqref{eq:w}
instead of \eqref{eq:u} is that solutions of \eqref{eq:w} are
classical and thus classical comparison principle applies.  We also
note that the stationary version of \eqref{eq:w}:
\begin{eqnarray}\label{eq:wst}
  \left\{
\begin{array}{ll}
  - W_{rr}-\frac{d-1}{r}W_r+W^p =0& r\in(R,\infty), \\
  W(R)=v_\alpha(R),&
\end{array}\right.
\end{eqnarray}
where, with the same abuse of notation, $v_\alpha(r)$ is the solution
of \eqref{eq:v} written as a function of the radial coordinate, has a
unique solution satisfying $W(r)\to 0$ as $r\to \infty$ (see, e.g.,
\cite[Theorem 4.3.2]{pucci}).  Therefore, $W(r)=v_\alpha(r)$ for $r\ge
R$.

The proof of Theorem \ref{t:2} is based on a construction of explicit
sub- and super-solutions for problem \eqref{eq:w}.  Let
\begin{eqnarray}
  N[w] := w_t-w_{rr}-\frac{d-1}{r}w_r+w^p.
\end{eqnarray}
We say that $\overline w$ is a super-solution of $w$ if it solves the
differential inequality
 \begin{eqnarray}\label{eq:wsup}
   \left\{
     \begin{array}{ll}
       N[\overline w]\ge 0 & (r,t)\in(R,\infty)\times(0,\infty) \\
       \overline w(R,t)\ge u(R,t)& t\in(0,\infty)\\
       \overline w(r,0)\ge 0 & r\in (R,\infty)
     \end{array}\right.
 \end{eqnarray}
 Similarly, we say that $\underline w$ is a sub-solution of $w$ if
 \begin{eqnarray}\label{eq:wsub}
 \left\{
 \begin{array}{ll}
 N[\underline w]\le 0 & (r,t)\in(R,\infty)\times(0,\infty) \\
 \underline w(R,t)\le u(R,t)& t\in(0,\infty)\\
 \underline w(r,0)\le 0 & r \in(R,\infty)
 \end{array}\right.
 \end{eqnarray}
 Then by classical comparison principle \cite{protter,pucci} the
 solution of problem \eqref{eq:w} obeys
 \begin{eqnarray}
   \underline w(r,t) \le w(r,t)\le \overline w(r,t) \quad
   (r,t)\in[R,\infty)\times[0,\infty). 
 \end{eqnarray}
 Likewise, we define radial sub- and super-solutions for
 \eqref{eq:wst}. Namely, let
\begin{eqnarray}\label{eq:M}
  M[W]=-W_{rr}-\frac{d-1}{r}W_r+W^p
\end{eqnarray}
Then the functions $\overline W$ and $\underline W$ are called super-
and sub-solutions of \eqref{eq:wst}, respectively, if they satisfy
\begin{eqnarray}\label{eq:wstsup}
  \left\{
 \begin{array}{ll}
   M[\overline W]\ge 0 & r\in(R,\infty) \\
   \overline W(R)\ge v_\alpha(R)& 
 \end{array}\right.
 \end{eqnarray}
and
 \begin{eqnarray}\label{eq:wstsub}
 \left\{
 \begin{array}{ll}
   M[\underline W]\le 0 & r\in(R,\infty) \\
   \underline W(R)\le v_\alpha(R)& 
 \end{array}\right.
 \end{eqnarray}
 respectively. Again, by the maximum principle \cite{pucci} we have
\begin{eqnarray}
  \underline{W}(r)\le W(r)\le \overline{W}(r) \quad r\in[R,\infty).
\end{eqnarray}

In the following lemma, we construct explicitly a sub- and a
super-solution for the stationary problem \eqref{eq:wst}.

\begin{lemma} \label{l:vr}
  Let $R>0$ and $\alpha>0$.  Then there exists $0<\bar
  \gamma<1$ such that for all $\bar \gamma \le \gamma < 1$ there
  exists $b>0$ such that the function
  \begin{eqnarray}\label{eq:v*}
    v_0(r):=\frac{c(p,d)}{(r+b r^{\gamma})^{\frac{2}{p-1}}}
  \end{eqnarray}
  is a sub-solution for problem \eqref{eq:wst}.  Moreover, the
  function $v_{\infty}(r)$ defined by \eqref{eq:vinf} is a
  super-solution for this problem.
\end{lemma}
\begin{proof}
  Let us start with the construction of a sub-solution. First, observe
  that direct substitution of $v_0$ defined by \eqref{eq:v*} into
  \eqref{eq:M} gives
  \begin{eqnarray}
    && M[v_0]=\frac{2bc(p,d)}{(p-1)^2(r+b
      r^{\gamma})^{\frac{2p}{p-1}}}\Bigg\{ \gamma
    b\left[(d-2)(p-1)-2\gamma \right]r^{2(\gamma-1)} \nonumber \\ 
    && +(p-1)\left[
      \gamma^2-\left(\frac{4p}{p-1}-d\right)\gamma+
      d-1\right]r^{\gamma-1}\Bigg\}.  
\end{eqnarray}
In order for $v_0$ to satisfy the condition $M[v_0]\le 0$, it is
sufficient to choose the parameter $\gamma$ such that the following
two inequalities hold:
\begin{eqnarray} \label{eq:gam1}
\gamma\ge \gamma_1:= \frac{(d-2)(p-1)}{2},
\end{eqnarray} 
and
\begin{eqnarray}\label{eq:gam2}
s(\gamma)\le 0,
\end{eqnarray}
where
\begin{eqnarray}\label{eq:gam22}
s(\gamma)= \gamma^2-\left(\frac{4p}{p-1}-d\right)\gamma+d-1.
\end{eqnarray}
Indeed, first observe that $\gamma_1=0$ for $d=2$ and
$\gamma_1=1-(d-2)(p^*-p)/2<1$ for $d\ge 3$.  Next, since $s(\gamma)$
is a parabola having the following properties: $s(0)=d-1 > 0$,
$s(1)=-\frac{2d}{p-1}\left(1-\frac{p}{p^*}\right)<0$ and $
s^{\prime}(\gamma)\le s^{\prime}(1)< -2/(p-1)<0$, we conclude that
$s(\gamma)\le 0$ for all $1>\gamma\ge \gamma_2$, where $\gamma_2<1$ is
the smallest root of the equation $s(\gamma)=0$.  Setting $\bar
\gamma=\max\{\gamma_1,\gamma_2\}$ we have $M[v_0]\le 0$.  Next, we
choose $b$ sufficiently large so that an inequality $v(R)\ge v_0(R)$
holds. Thus, both conditions in \eqref{eq:wstsub} are satisfied, and
$v_0(r)$ is a sub-solution for problem \eqref{eq:wst}.

Now, let us verify that $v_{\infty}(r)$ is a super-solution for
$W$. First, by direct substitution we have $M[v_{\infty}]=0$.
Moreover, $v_{\infty}(r)>v_{\alpha}(r)$ for any $\alpha$ and $r>0$
(see section 2). Thus, $v_{\infty}(r)$ is a super-solution.
\end{proof}

\medskip

\begin{remark}
  It follows from Lemma \ref{l:vr} that for each pair $R>0$,
  $\alpha>0$ one can choose $0<\gamma<1$ and $b > 0$ so that
  $v_0(r)<v_{\alpha}(r)<v_{\infty}(r)$ for all $r\ge R$.
\end{remark}

\medskip

Now we construct a super-solution for problem \eqref{eq:w}.

\begin{lemma}\label{lem:super}
  There exists $T_1>0$ such that the function
  \begin{eqnarray}\label{eq:bw}
    \overline w(r,t):=v_{\infty}(r)\phi(y), \quad
    y=\log\left(\frac{r}{\sqrt{t+T_1}}\right) 
  \end{eqnarray} 
  is a super-solution of \eqref{eq:w}
\end{lemma}
\begin{proof}
  First, observe that $\overline w(r,t)=U(r,t+T_1)$ and thus verifies
  $N[\overline w]=0$.  Next, since $v_{\infty}(R)> v(R)>u(R,t)$ and
  $\phi(\zeta)\to 1$ as $\zeta \to -\infty$, one can always choose
  $T_1$ large enough, so that
  $v_{\infty}(R)\phi(\log\left(R/\sqrt{t+T_1}\right))>u(R,t)$ for all
  $t>0$.  As a result, $\overline w$ is a super-solution.
\end{proof}

Finally, we give a construction of a sub-solution for problem
\eqref{eq:w}.  This construction requires the following elementary
lemma.

\begin{lemma}\label{l:100}
  Let $\phi(\zeta)$ be as in Theorem \ref{t:exun}. Then, there exists
  $\zeta_0>0$ such that $\phi(\zeta)\le |\phi^{\prime}(\zeta)|$ for
  all $\zeta \ge \zeta_0$.
\end{lemma} 
\begin{proof}
  First, rewrite \eqref{eq:phiz} as follows:
  \begin{align}
    \left[ \phi^{\prime}+\left(\frac{e^{2\zeta}}{2}-\frac{p+3}{p-1}
        + d-1 \right)\phi \right]^{\prime}=  
    \left[ e^{2\zeta}-\frac{2}{p-1}\left(\frac{p+1}{p-1}-
        d + 1 \right)(1-\phi^{p-1}) \right] \phi.  
\end{align}
Integrating this equation from $\zeta$ to $\infty$ and taking into
account \eqref{eq:phibcz2} and \eqref{eq:asympt}, we have
\begin{eqnarray}
  -\phi^{\prime}(\zeta)-\left(\frac{e^{2\zeta}}{2}-
    \frac{p+3}{p-1}+d-1 \right)\phi(\zeta)>0,
\end{eqnarray}
provided that $\zeta$ is sufficiently large. In view of the fact that
$\phi^\prime(\zeta)<0$ (see Theorem \ref{t:exun}), we obtain
 \begin{eqnarray}
 \phi(\zeta)\le 4e^{-2\zeta} |\phi^{\prime}(\zeta)|,
 \end{eqnarray}
 for large enough $\zeta$, which gives the desired result.
\end{proof}

\begin{lemma}\label{lem:sub}  
  Given $R>0$ and $\alpha>0$, there exist $T_2>0$, $b>0$ and
  $0<\gamma<1$ such that the function
\begin{eqnarray}\label{eq:vp}
\underline w(r,t)=\left\{\begin{array}{ll}
0 & t<T_2, \\
v_0(r)\phi(z) & t\ge T_2,
\end{array}\right.
\end{eqnarray}
where $v_0$ is given by \eqref{eq:v*} and
\begin{eqnarray}
  z=\ln\left(\frac{r+br^{\gamma}}{\sqrt{t-T_2}}\right),
\end{eqnarray}
is a sub-solution for problem \eqref{eq:w}.
\end{lemma}
\begin{proof}
  Observe first that by \eqref{eq:asympt} and the fact that $\phi\in
  C^{\infty}(\mathbb{R})$ the function defined by \eqref{eq:vp} is
  smooth for all $r\ge R$ and $t>0$.  Next, let us show that
  $N[\underline w]\le 0$. For $t\le T_2$ this is trivial. For $t>T_2$,
  after a very tedious but straightforward computation, we have:
  \begin{eqnarray}
    N[\underline
    w]=-\frac{c(p,d)b}{(r+br^\gamma)^{\frac{2p}{p-1}}}\left[(-\phi^{\prime})(z)
      A_1(r,z)+\phi(z) A_2(r,z)\right], 
  \end{eqnarray}
  where
  \begin{eqnarray}
    && A_1(r,z)= \left(\gamma
      e^{2z}-(d-1-\gamma)\delta\right)r^{\gamma-1} 
    +b\left(\frac{\gamma^2}{2}e^{2z}-(d-2)\delta\right)r^{2(\gamma-1)},
    \\ 
    &&A_2(r,z)=\frac{2}{p-1}\Big\{\left[2\gamma\left(\frac{p+1}{p-1}
        -d+1 \right)\phi^{p-1}-(d-1-\gamma)\delta\right]r^{\gamma-1}\nonumber\\  
    &&+b\left[\gamma^2\left(\frac{p+1}{p-1}-d+1\right)\phi^{p-1}
      -(d-2)\delta\right] r^{2(\gamma-1)}\Big\},
\end{eqnarray}
and 
\begin{eqnarray}
\delta=1-\gamma.
\end{eqnarray}
Thus, all we need to show is that $\Psi(r,z)=(-\phi^{\prime})(z)
A_1(r,z)+\phi(z) A_2(r,z)\ge 0$.  In order to show that this condition
indeed holds, assume first that $z<\zeta_0$, where $\zeta_0$ is chosen
as in Lemma \ref{l:100}. In this case, by Theorem \ref{t:exun} we have
$\phi\ge k_1>0$ and $|\phi^{\prime}|\le k_2$ for some positive
constants $k_1$ and $k_2$.  Therefore
\begin{eqnarray}
  && \Psi(r,z)\ge
  \frac{2}{p-1}\left[2\gamma\left(\frac{p+1}{p-1}-
      d+1 \right)k_1^{p-1}-\left(1+\frac{k_2(p-1)}{2}\right)
    (d-1-\gamma)\delta\right]r^{\gamma-1}+  
  \nonumber \\ 
  &&\frac{2b}{p-1}\left[\gamma^2\left(\frac{p+1}{p-1}
      -d + 1 \right)k_1^{p-1}-\left(1+\frac{k_2(p-1)}{2}
    \right)(d-2)\delta\right]  
  r^{2(\gamma-1)},
\end{eqnarray}
and thus $\Psi(r,z)\ge 0$ independently of $b$ and $T_2$ for all $r\ge
R$ and $z<\zeta_0$, provided that $\delta$ is sufficiently small.  In
the case when $z\ge \zeta_0$ we have, by Lemma \ref{l:100}, that
$\phi<|\phi^{\prime}|$, which implies
\begin{eqnarray}
  && \Psi(r,z)\ge  |\phi^{\prime}(z)|\Bigg[\left(\gamma
    -\left(1+\frac{2}{p-1}\right)(d-1-\gamma)\delta\right)r^{\gamma-1} \nonumber \\ 
  &&+b\left(\frac{\gamma^2}{2}-\left(1+\frac{2}{p-1}\right)
    (d-2)\delta\right)r^{2(\gamma-1)}\Bigg],
\end{eqnarray}
and thus $\Psi(r,z)\ge 0$ for all $r\ge R$ and $z\ge \zeta_0$ for
$\delta$ small enough independently of $b$ and $T_2$.  Therefore, we
conclude that one can choose $\gamma<1$ sufficiently close to unity so
that $N[\underline w]\le 0$ for all $r\ge R$ and $t>0$ independently
of $b$ and $T_2$.

Now we choose $T_2$ and $b$ large enough so that $u(R,t)>2v(R)/3$ for
$t>T_2$ and $v_0(R)=v(R)/2$, which is always possible, since by
Theorem \ref{t:eu} for each $R>0$ the function $u(R,t)$ is a monotone
increasing function of $t$ and $\lim_{t\to\infty} u(R,t)=v(R)$, and
$v_0$ is a decreasing function of $b$ that vanishes as
$b\to\infty$. In view of this observation and the fact that $\phi<1$,
we have
\begin{eqnarray}
\underline w(R,t)<u(R,t),
\end{eqnarray}
for all $t > 0$. As a result, $\underline w$ given by \eqref{eq:vp} is
a sub-solution for \eqref{eq:w}.
\end{proof}

\medskip

\emph{Proof of Theorem \ref{t:2}.}  Recall first that by construction
\begin{eqnarray}
  \underline w(r,t)\le u(r,t)\le \overline w(r,t),
\end{eqnarray}
for all $t\ge 0$ and $r\ge R$. Therefore, by Lemmas \ref{lem:sub} and
\ref{lem:super}, for $t>T_2$ and $r>R$ we have
\begin{eqnarray}\label{eq:tmm}
\underline w(r,t)/v_{\alpha}(r)\le F(\zeta,t)\le \overline w(r,t)/v_{\alpha}(r),
\end{eqnarray}
where $F$ and $\zeta$ and defined by \eqref{eq:tm00}. Using this
inequality, after straightforward algebraic computations we obtain
\begin{eqnarray}
  &&F(\zeta,t)\le
  \frac{v_{\infty}(r)}{v_{\alpha}(r)}\phi\left(\zeta-\frac12
    \ln\left(1+T_1/t\right)\right),  \label{eq:tm1a} 
  \\ 
  && F(\zeta,t)\ge \frac{v_0(r)}{v_{\alpha}(r)}\phi\left[ \zeta
    +\ln
    \left(\frac{1+be^{-(1-\gamma)\zeta}
        /t^{(1-\gamma)/2}}{\sqrt{1-T_2/t}}\right)\right].     \label{eq:tm1} 
\end{eqnarray}
It is also not difficult to see that
\begin{eqnarray}\label{eq:tm2}
  H(\zeta,t):=v_0(r)/v_{\infty}(r)=\frac{1}{1+
    b e^{-(1-\gamma)\zeta}/t^{(1-\gamma)/2}}.  
\end{eqnarray}
Taking the limit as $t\to\infty$ in \eqref{eq:tm1a}, \eqref{eq:tm1}
and \eqref{eq:tm2}, we obtain \eqref{eq:tm01}.  \qed

\section*{Acknowledgements} This work was supported, in part, by NSF
via grant DMS-1119724. CBM was also partially supported by NSF via
grants DMS-0908279 and DMS-1313687.  The work of PVG was also
partially supported by a grant 317882 from the Simons Foundation, and
by the US-Israel Binational Science Foundation via grant 2012057.  The
authors are grateful to V. V. Matveev and V. Moroz for valuable
discussions, and to S. Y. Shvartsman for suggesting this problem to
us.

\bibliographystyle{siam}


\bibliography{selfconv2}
\end{document}